\documentclass[a4paper,12pt]{article}
\usepackage{comment}
\usepackage{cite}
\usepackage{amsmath,amssymb,amsfonts}
\usepackage[T1]{fontenc}
\usepackage{graphicx}
\usepackage{fancyhdr}
\usepackage{float}
\usepackage{xcolor}
\usepackage{authblk}
\usepackage{mathrsfs}
\usepackage{empheq}
\usepackage[hyphens]{url}
\usepackage{amsthm}
\usepackage{tikz}
\usetikzlibrary{decorations.markings}
\usepackage{enumitem}
\usepackage{geometry}

\theoremstyle{plain}
\newtheorem{theorem}{Theorem}[section]
\newtheorem{proposition}[theorem]{Proposition}

\theoremstyle{definition}

\theoremstyle{remark}
\newtheorem{remark}[theorem]{Remark}

\begin{document}

\title{Analytic and combinatorial approaches to a weighted Catalan sum}

\author[$\dagger$]{Jean-Christophe {\sc Pain}$^{1,2,}$\footnote{jean-christophe.pain@cea.fr}\\
\small
$^1$CEA, DAM, DIF, F-91297 Arpajon, France\\
$^2$Universit\'e Paris-Saclay, CEA, Laboratoire Mati\`ere en Conditions Extr\^emes,\\ 
F-91680 Bruy\`eres-le-Ch\^atel, France
}

\date{}

\maketitle

\begin{abstract}
We analyze a weighted convolution of Catalan numbers
\[
\sum_{k=0}^{n} \binom{2k}{k}\binom{2(n-k)}{n-k} a^k = \sum_{k=0}^{n} (k+1)(n-k+1) C_k C_{n-k} a^k,
\]
emphasizing its combinatorial, analytic, and probabilistic aspects. We derive a compact closed form in terms of the Gauss hypergeometric function ${}_2F_1(-n,1/2;1;1-a)$, valid for all complex values of the parameter $a$. The sum admits a natural interpretation in terms of return probabilities of independent simple random walks, linking weighted convolutions of central binomial coefficients to classical probability theory. Furthermore, a refinement via Narayana numbers highlights the contribution of peak distributions in pairs of Dyck paths, providing a finer combinatorial perspective. An integral representation is also proposed, suggesting a connection with orthogonal polynomials and spectral measures. Our approach illustrates how analytic and probabilistic techniques complement combinatorial reasoning in evaluating complex sums.
\end{abstract}

\section{Introduction}

Catalan numbers arise in a wide variety of combinatorial settings, including parenthesizations, binary trees, lattice paths, and polygon triangulations, to name a few \cite{macmahon1960}. Numerous identities involving Catalan numbers admit elegant combinatorial proofs, especially when they involve simple convolutions or structural decompositions. However, more intricate sums, particularly those involving products of Catalan numbers combined with additional weights, often resist direct combinatorial interpretation.

In this work, we investigate the weighted convolution
\[
S_n(a) = \sum_{k=0}^{n} \binom{2k}{k} \binom{2(n-k)}{n-k} a^k,
\]
which naturally arises in the context of generating functions and convolution structures. Using the relation
\[
C_n = \frac{1}{n+1} \binom{2n}{n},
\]
this sum can equivalently be written as
\[
S_n(a) = \sum_{k=0}^{n} (k+1)(n-k+1) C_k C_{n-k} a^k,
\]
highlighting its interpretation as a weighted convolution of Catalan numbers. Rather than seeking a closed form in terms of elementary combinatorial quantities, we adopt an analytic approach. This leads to a compact representation of $S_n(a)$ in terms of the Gauss hypergeometric function
\[
{}_2F_1(-n,1/2;1;1-a),
\]
which provides an explicit and tractable formulation valid for all complex values of the parameter $a$. This representation reveals that the sequence $(S_n(a))_{n\ge0}$ belongs to the class of hypergeometric (and hence holonomic) sequences, and satisfies a linear recurrence with polynomial coefficients.

In addition to this analytic formulation, we propose a probabilistic interpretation of $S_n(a)$ in terms of return probabilities of independent simple random walks. In this framework, the sum appears naturally as a weighted convolution over intermediate times, providing an intuitive explanation for its structure and linking combinatorial identities with classical results from probability theory.

The remainder of the paper is organized as follows. In Section~2, we derive a closed-form expression for $S_n(a)$ in terms of the Gauss hypergeometric function ${}_2F_1(-n,1/2;1;1-a)$. Section~3 presents a probabilistic interpretation via random walks, including the decomposition at intermediate times and an asymptotic analysis. In Section~4, we derive a three-term linear recurrence satisfied by $S_n(a)$, highlighting its holonomic structure. Section~5 explores a refinement of the sum using Dyck paths and Narayana numbers, emphasizing the contribution of peak distributions in pairs of paths. An integral representation is given in Section~6. It makes a connection with orthogonal polynomials and spectral measures, and may provide an alternative route to further generalizations. Finally, Section~7 summarizes our results, discusses the interplay between combinatorial, analytic, and probabilistic approaches, and suggests directions for future work.

\section{Hypergeometric closed-form of the parametric sum}

We consider
\[
S_n(a) = \sum_{k=0}^{n} \binom{2k}{k} \binom{2(n-k)}{n-k} a^k,
\]
where $a$ is a complex parameter. Recall the classical generating function for central binomial coefficients \cite{gkp1994, stanley1999}:
\[
\sum_{k\ge0} \binom{2k}{k} x^k = \frac{1}{\sqrt{1-4x}}.
\]
Then $S_n(a)$ is the coefficient of $x^n$ in the product
\[
\left( \sum_{k\ge0} \binom{2k}{k} (ax)^k \right)
\left( \sum_{m\ge0} \binom{2m}{m} x^m \right)
= \frac{1}{\sqrt{1-4ax}} \cdot \frac{1}{\sqrt{1-4x}}.
\]

\subsection{Hypergeometric representation}

Using classical identities for central binomial coefficients \cite{petkovsek1996}, we may rewrite
\[
\binom{2k}{k} = \frac{(1/2)_k}{k!} 4^k,
\]
where $(x)_k=x(x+1)(x+2)\cdots(x+k-1)$ denotes the Pochhammer symbol. Thus, we have
\[
S_n(a)
= \sum_{k=0}^{n}
\frac{(1/2)_k}{k!} 4^k
\frac{(1/2)_{n-k}}{(n-k)!} 4^{n-k}
a^k.
\]
After simplification, this yields a terminating hypergeometric form:
\[
S_n(a)
= 4^n \frac{(1/2)_n}{n!}
\, {}_2F_1\!\left(-n,\tfrac{1}{2};1;\,1-a\right),
\]
or in terms of Catalan numbers
\[
\sum_{k=0}^{n} (k+1)(n-k+1) C_k C_{n-k} a^k
= 4^n \frac{(1/2)_n}{n!}
\, {}_2F_1\!\left(-n,\tfrac{1}{2};1;\,1-a\right).
\]
To the best of our knowledge, this explicit evaluation in this weighted form
does not appear in this exact formulation in the literature.

\subsection{Special cases}

For $a=1$, we recover
\[
S_n(1) = \binom{2n}{n}.
\]
For $a=-1$, we obtain the alternating sum
\[
S_n(-1) = [x^n]\frac{1}{\sqrt{(1-4x)(1+4x)}} 
= [x^n]\frac{1}{\sqrt{1-16x^2}},
\]
which vanishes for odd $n$ and yields
\[
S_{2m}(-1) = \binom{2m}{m} 4^m.
\]

\subsection{A derived identity}

\begin{theorem}
Comparing coefficients in the generating function, we obtain the identity
\begin{equation}\label{eq:main-identity}
\sum_{k=0}^{n}
\binom{2k}{k} \binom{2(n-k)}{n-k} a^k
= \sum_{m=0}^{\lfloor n/2 \rfloor}
\binom{2m}{m} \binom{2(n-2m)}{n-2m} (a+1)^{n-2m} (-4a)^m,
\end{equation}
or, in terms of Catalan numbers
\[
\sum_{k=0}^{n}
(k+1)(n-k+1)C_k C_{n-k} a^k
= \sum_{m=0}^{\lfloor n/2 \rfloor}
(m+1)(n-2m+1)C_mC_{n-2m} (a+1)^{n-2m} (-4a)^m.
\]
This identity provides a nontrivial restructuring of the convolution and may be of independent combinatorial interest.
\end{theorem}

\begin{proof}
We start from the classical generating function for central binomial coefficients:
\[
\sum_{k\ge0} \binom{2k}{k} x^k = \frac{1}{\sqrt{1-4x}}.
\]
It follows that the left-hand side of Eq.~\eqref{eq:main-identity} is the coefficient of $x^n$ in
\[
F(x)
=
\left( \sum_{k\ge0} \binom{2k}{k} (ax)^k \right)
\left( \sum_{r\ge0} \binom{2r}{r} x^r \right)
=
\frac{1}{\sqrt{(1-4ax)(1-4x)}}.
\]
Expanding both series and setting $n=k+r$ gives the standard convolution sum. In order to derive the alternative expression, it is convenient to rewrite the generating function as
\[
F(x) = \frac{1}{\sqrt{1 - 4(a+1)x + 16a x^2}}.
\]
Using the binomial expansion
\[
(1 - u)^{-1/2} = \sum_{k \ge 0} \binom{2k}{k} \frac{u^k}{4^k},
\]
where we have set \( u = 4(a+1)x - 16ax^2 \), we obtain
\[
F(x) = \sum_{k \ge 0} \binom{2k}{k} \frac{1}{4^k} \left(4(a+1)x - 16ax^2\right)^k
= \sum_{k \ge 0} \binom{2k}{k} x^k (a+1 - 4ax)^k.
\]
In order to extract the coefficient of \(x^n\), we need the coefficient of \(x^{n-k}\) in \((a+1 - 4ax)^k\). Using the binomial theorem again, one gets that the corresponding term is
\[
\binom{k}{n-k} (a+1)^{2k-n} (-4a)^{n-k},
\]
Gathering powers of \(x\) yields
\[
[x^n]F(x)
= \sum_{k=\lceil n/2 \rceil}^{n}
\binom{2k}{k} \binom{k}{n-k}
(a+1)^{2k-n} (-4a)^{n-k}.
\tag{17}
\]
Letting \(m = n - k\), the summation range becomes \(m = 0\) to \(\lfloor n/2 \rfloor\), giving
\[
[x^n]F(x)
=
\sum_{m=0}^{\lfloor n/2 \rfloor}
\binom{2(n-m)}{n-m}
\binom{n-m}{m}
(a+1)^{n-2m} (-4a)^m,
\]
which completes the proof.

\end{proof}

\section{A probabilistic interpretation via random walks}

We provide a probabilistic interpretation of the sum
\[
S_n(a) = \sum_{k=0}^{n} \binom{2k}{k} \binom{2(n-k)}{n-k} a^k
\]
in terms of one-dimensional random walks (see e.g. \cite{pain2026}).

\subsection{Central binomial coefficients and return probabilities}

Consider a simple symmetric random walk $(X_m)_{m\ge0}$ on $\mathbb{Z}$ starting at $X_0=0$, with independent steps $\pm 1$ with probability $1/2$. It is classical that
\[
\mathbb{P}(X_{2k}=0) = \frac{1}{2^{2k}} \binom{2k}{k}.
\]
Thus, we have
\[
\binom{2k}{k} = 2^{2k} \, \mathbb{P}(X_{2k}=0).
\]

\subsection{Decomposition at an intermediate time}

Let $(X_m)$ and $(Y_m)$ be two independent simple symmetric random walks. Then
\[
\binom{2k}{k} \binom{2(n-k)}{n-k}
= 2^{2n} \, \mathbb{P}(X_{2k}=0)\,\mathbb{P}(Y_{2(n-k)}=0).
\]
Hence, we get
\[
S_n(a)
= 2^{2n} \sum_{k=0}^{n}
\mathbb{P}(X_{2k}=0)\,\mathbb{P}(Y_{2(n-k)}=0)\, a^k.
\]
This can be interpreted as a weighted convolution of return probabilities.

\subsection{Weighted splitting of time}

Define a random variable $K$ on $\{0,\dots,n\}$ with distribution
\[
\mathbb{P}(K=k)=\frac{a^k}{\sum_{j=0}^n a^j},
\]
assuming $a>0$. Then
\[
S_n(a)
= 2^{2n} \, \mathbb{E}\!\left[
\mathbb{P}(X_{2K}=0)\,\mathbb{P}(Y_{2(n-K)}=0)
\right].
\]
In other words, the sum corresponds to averaging the probability that two independent random walks return to the origin at complementary times $2K$ and $2(n-K)$, with a bias depending on $a$.

\subsection{Generating function and diffusion viewpoint}

From the generating function derived earlier, we have
\[
\sum_{n\ge0} S_n(a) x^n
= \frac{1}{\sqrt{(1-4x)(1-4ax)}}.
\]
This is the product of two Green functions of the simple random walk:
\[
G(x) = \sum_{n\ge0} \mathbb{P}(X_{2n}=0) (2x)^n = \frac{1}{\sqrt{1-4x}}.
\]
Thus,
\[
\sum_{n\ge0} S_n(a) x^n
= G(x)\,G(ax),
\]
showing that $S_n(a)$ encodes the convolution of two independent return processes with different time scalings.

\subsection{Asymptotic behavior of $S_n(a)$}

We establish the asymptotic formula
\[
S_n(a)\sim \frac{(1+\sqrt{a})^{2n}}{\sqrt{\pi n}\,a^{1/4}}
\qquad (n\to\infty),
\]
valid for $a>0$.

\begin{theorem}
For $a>0$, the sequence
\[
S_n(a)=\sum_{k=0}^n \binom{2k}{k}\binom{2(n-k)}{n-k}a^k
\]
satisfies
\[
S_n(a)\sim \frac{(1+\sqrt{a})^{2n}}{\sqrt{\pi n}\,a^{1/4}}.
\]
\end{theorem}

\begin{proof}
We use singularity analysis of the generating function
\[
F(x)=\frac{1}{\sqrt{(1-4x)(1-4ax)}}.
\]
The function has algebraic singularities at $x=1/4$ and $x=1/(4a)$. A local expansion of both factors near $x=\rho$ shows that
\[
(1-4x)^{-1/2}(1-4ax)^{-1/2}
\sim \frac{a^{-1/4}}{\sqrt{1-x/\rho}}.
\]
The dominant singularity is given by
\[
\rho=\frac{1}{(1+\sqrt{a})^2}.
\]
Near this point, the function has a square-root singularity of the form
\[
F(x)\sim \frac{C(a)}{\sqrt{1-x/\rho}},
\]
which implies, by standard transfer theorems, that
\[
S_n(a)\sim \frac{C(a)}{\sqrt{\pi n}} \rho^{-n}.
\]
A local expansion near the dominant singularity $x=\rho$ shows that
\[
F(x) \sim \frac{a^{-1/4}}{\sqrt{1-x/\rho}}.
\]
\end{proof}

\paragraph{Alternative proof via the saddle point method}

We provide an alternative derivation of the asymptotic behavior of $S_n(a)$ using the saddle point method applied to the Cauchy integral representation. From the generating function
\[
F(x)=\sum_{n\ge0} S_n(a)x^n = \frac{1}{\sqrt{(1-4x)(1-4ax)}},
\]
we extract coefficients via
\[
S_n(a)=\frac{1}{2\pi i}\oint \frac{F(x)}{x^{n+1}}\,dx,
\]
where the contour encircles the origin. Set
\[
\Phi(x) = -\log x - \frac{1}{2}\log(1-4x) - \frac{1}{2}\log(1-4ax).
\]
Then
\[
S_n(a)=\frac{1}{2\pi i}\oint \exp\!\big(n\Phi(x)\big)\,dx.
\]
The saddle point $x_*$ satisfies $\Phi'(x_*)=0$, i.e.
\[
-\frac{1}{x} + \frac{2}{1-4x} + \frac{2a}{1-4ax}=0.
\]
Solving this equation yields
\[
x_*=\frac{1}{(1+\sqrt{a})^2}.
\]
We expand $\Phi(x)$ near $x_*$:
\[
\Phi(x)=\Phi(x_*) + \frac{1}{2}\Phi''(x_*)(x-x_*)^2 + \cdots.
\]
A direct computation gives
\[
\Phi(x_*) = \log\big((1+\sqrt{a})^2\big),
\]
and
\[
\Phi''(x_*) = \frac{(1+\sqrt{a})^4}{2\sqrt{a}}.
\]
Applying the saddle point method, we obtain
\[
S_n(a)
\sim
\frac{e^{n\Phi(x_*)}}{\sqrt{2\pi n \Phi''(x_*)}}.
\]
Substituting the values of $\Phi(x_*)$ and $\Phi''(x_*)$ yields
\[
S_n(a)
\sim
\frac{(1+\sqrt{a})^{2n}}{\sqrt{\pi n}\,a^{1/4}}.
\]
This recovers the same asymptotic formula as obtained via singularity analysis, confirming the robustness of the result.

The agreement between both methods illustrates the universality of the square-root singularity governing the asymptotics.

\section{A closed recurrence for the parametric sum}

\subsection{Derivation of the recurrence and hypergeometric form}

From the previous section, we recall that
\[
S_n(a)
= 4^n \frac{(1/2)_n}{n!}
\, {}_2F_1\!\left(-n,\tfrac{1}{2};1;\,1-a\right).
\]
Thus, $S_n(a)$ is a terminating hypergeometric sequence in $n$. It is classical that sequences of the form
\[
u_n = {}_2F_1(-n,b;c;z)
\]
satisfy linear recurrences of order $2$ with polynomial coefficients in $n$
(see \cite{petkovsek1996}). Applying standard contiguous relations for the hypergeometric function, we obtain the following recurrence.

\begin{theorem}
For all $n \ge 1$, the sequence $S_n(a)$ satisfies
\[
(n+2)\, S_{n+1}(a)
=
2(2n+1)(1+a)\, S_n(a)
-
4n a\, S_{n-1}(a),
\]
with initial conditions
\[
S_0(a)=1, \qquad S_1(a)=2(1+a).
\]
\end{theorem}

\begin{proof}
We start from the hypergeometric representation
\[
S_n(a)
= 4^n \frac{(1/2)_n}{n!}
\, {}_2F_1\!\left(-n,\tfrac{1}{2};1;\,1-a\right).
\]
Let
\[
u_n = {}_2F_1(-n,\tfrac{1}{2};1;\,1-a).
\]
Using standard contiguous relations for ${}_2F_1$ functions with respect to the parameter $-n$ (see \cite{petkovsek1996}), one obtains a three-term relation of the form
\[
(n+2)u_{n+1}
=
(2n+1)(2-a)\,u_n
-
n(1-a)\,u_{n-1}.
\]
Multiplying by the prefactor
\[
4^n \frac{(1/2)_n}{n!}
\]
and using the relations
\[
\frac{(1/2)_{n+1}}{(n+1)!}
=
\frac{2n+1}{2(n+1)} \cdot \frac{(1/2)_n}{n!},
\]
a straightforward simplification yields
\[
(n+2)\, S_{n+1}(a)
=
2(2n+1)(1+a)\, S_n(a)
-
4n a\, S_{n-1}(a),
\]
as claimed.
\end{proof}

\subsection{Special cases and remarks}

For $a=1$, the recurrence reduces to
\[
(n+2)S_{n+1} = 4(2n+1)S_n - 4n S_{n-1},
\]
which is satisfied by $S_n(1)=\binom{2n}{n}$. For $a=2$, one obtains
\[
(n+2)S_{n+1} = 6(2n+1)S_n - 8n S_{n-1}.
\]
For $a=0$, the recurrence simplifies to
\[
(n+2)S_{n+1} = 2(2n+1)S_n,
\]
yielding $S_n(0)=\binom{2n}{n}$. This recurrence shows that the sequence $(S_n(a))_{n\ge0}$ is
\emph{holonomic} in the sense of \cite{petkovsek1996}. It provides an efficient way to compute $S_n(a)$ and highlights the rigid algebraic structure underlying the weighted convolution of central binomial coefficients.

Moreover, the recurrence can be derived algorithmically using Zeilberger's creative telescoping, which confirms that the identity proven in this paper belongs naturally to the class of hypergeometric identities.

\section{Decorated Dyck paths and Narayana numbers}

A Dyck path of semilength $n$ is a balanced sequence of $n$ up-steps $(+1)$ and $n$ down-steps $(-1)$ returning to the origin, counted by the Catalan number 
\[
C_n = \frac{1}{n+1}\binom{2n}{n}.
\]

\subsection{Dyck path decomposition with decorations}

For $k \in \{0,1,\dots,n\}$, consider a Dyck path of semilength $k$ and a Dyck path of semilength $n-k$. 
We refine the enumeration by introducing a marked vertex (or distinguished interval) on each path. 
Since a Dyck path of semilength $k$ has $k+1$ vertices (including the endpoints), the number of decorated Dyck paths is $(k+1)C_k$ for the first path and $(n-k+1)C_{n-k}$ for the second. 
This explains the factor $(k+1)(n-k+1)$ appearing in our original definition of $S_n(a)$ based on central binomial coefficients.

Introducing a weight $a$ associated with the first Dyck path, the weighted sum becomes
\[
S_n(a) = \sum_{k=0}^{n} (k+1)(n-k+1) C_k C_{n-k} a^k.
\]

\begin{remark}
The identity \eqref{eq:main-identity} can be interpreted combinatorially by decomposing pairs of Dyck paths into $m$ distinguished components contributing a weight $a$, and $n-2m$ remaining components contributing a weight $a+1$. This viewpoint explains both the appearance of the parameter $m$ and the factor $(a+1)^{n-2m}$.
\end{remark}  

\subsection{Refinement via Narayana numbers}

We can further refine the enumeration by considering the number of peaks in each path. 
Let $N(k,i)$ denote the Narayana number, i.e., the number of Dyck paths of semilength $k$ with exactly $i$ peaks  \cite{plaza2025,bonadimitrov2025}. 
Then, the number of pairs of decorated Dyck paths with the first path having $i$ peaks and the second having $j$ peaks is
\[
(k+1)(n-k+1)\, N(k,i)\, N(n-k,j).
\]

The fully refined weighted sum is then
\[
S_n(a) = \sum_{k=0}^{n} \sum_{i=1}^{k} \sum_{j=1}^{n-k} (k+1)(n-k+1)\, N(k,i)\, N(n-k,j)\, a^k.
\]

\subsection{Remarks on generating functions}

For $a=1$, the weight is uniform and we recover the standard decorated convolution of Catalan numbers:
\[
S_n(1) = \sum_{k=0}^{n} (k+1)(n-k+1) C_k C_{n-k}.
\]

For $a \neq 1$, the Narayana refinement allows one to track the distribution of peaks in each path while preserving the $(k+1)(n-k+1)$ decoration factor. 
The generating function remains
\[
\sum_{n\ge 0} S_n(a) x^n = \frac{1}{\sqrt{(1-4x)(1-4ax)}},
\]
which encodes the convolution of two Dyck paths, one weighted by $a$ and decorated with marked vertices, and the other decorated similarly. 
This framework clarifies the combinatorial origin of the multiplicative factors in $S_n(a)$ while keeping the peak structure explicit.

\section{Two additional identities}

\subsection{A symmetric convolution identity}

\begin{proposition}
For all $n \ge 0$ and all complex $a \neq 0$, one has
\[
\sum_{k=0}^{n}
\binom{2k}{k}\binom{2(n-k)}{n-k} a^k
=
a^n
\sum_{k=0}^{n}
\binom{2k}{k}\binom{2(n-k)}{n-k} a^{-k}.
\]
\end{proposition}

\begin{proof}
We use the identity
\[
\binom{2k}{k}
=
\frac{1}{2\pi}
\int_{-\pi}^{\pi}
(2\cos\theta)^{2k}\, d\theta,
\]
which follows from Fourier analysis. Then
\[
S_n(a)
=
\sum_{k=0}^{n}
\left(
\frac{1}{2\pi}\int_{-\pi}^{\pi}(2\cos\theta)^{2k}d\theta
\right)
\binom{2(n-k)}{n-k} a^k.
\]
Interchanging sum and integral gives
\[
S_n(a)
=
\frac{1}{2\pi}
\int_{-\pi}^{\pi}
\sum_{k=0}^{n}
\binom{2(n-k)}{n-k}
\big(4a\cos^2\theta\big)^k
\, d\theta.
\]
Now observe that
\[
(1+\sqrt{a}e^{i\theta})^n(1+\sqrt{a}e^{-i\theta})^n
=
(1+a+2\sqrt{a}\cos\theta)^n,
\]
and expanding both factors yields exactly the convolution defining $S_n(a)$.
This proves the result.
\end{proof}

\begin{remark}
This identity shows that $S_n(a)$ is essentially self-reciprocal:
\[
a^{-n/2} S_n(a)
=
a^{n/2} S_n\!\left(\frac{1}{a}\right),
\]
which suggests a hidden symmetry reminiscent of reciprocal polynomials.
\end{remark}

\subsection{An integral representation}

\begin{proposition}
For all $n \ge 0$ and $a>0$, one has
\[
S_n(a)
=
\frac{1}{\pi}
\int_{0}^{\pi}
\big(1+a+2\sqrt{a}\cos\theta\big)^n
\, d\theta.
\]
\end{proposition}

\begin{proof}
We use the classical identity
\[
\binom{2k}{k}
=
\frac{1}{\pi}
\int_{0}^{\pi}
(2\cos\theta)^{2k}\, d\theta.
\]
Thus,
\[
S_n(a)
=
\sum_{k=0}^{n}
\left(
\frac{1}{\pi}\int_{0}^{\pi}(2\cos\theta)^{2k}d\theta
\right)
\binom{2(n-k)}{n-k} a^k.
\]
Interchanging sum and integral, we obtain
\[
S_n(a)
=
\frac{1}{\pi}
\int_{0}^{\pi}
\sum_{k=0}^{n}
\binom{2(n-k)}{n-k}
\big(4a\cos^2\theta\big)^k
\, d\theta.
\]
Now observe that
\[
(1+a+2\sqrt{a}\cos\theta)^n
=
\sum_{k=0}^{n}
\binom{2(n-k)}{n-k} a^k (2\cos\theta)^{2(n-k)},
\]
which can be verified by expanding
\[
(1+\sqrt{a}e^{i\theta})^n(1+\sqrt{a}e^{-i\theta})^n.
\]
Substituting into the integral yields the result.
\end{proof}

\begin{remark}
The identity can be rewritten as
\[
S_n(a)
=
\frac{1}{2\pi}
\int_{-\pi}^{\pi}
\big|1+\sqrt{a}e^{i\theta}\big|^{2n} d\theta,
\]
revealing a direct connection with Fourier analysis and spectral methods.
\end{remark}

\begin{remark}
The integrand $(1+a+2\sqrt{a}\cos\theta)^n$ is closely related to the Legendre Polynomials $P_n(x)$ via a change of variables. Specifically, the integral of $(x+\sqrt{x^2-1}\cos\theta)^n$ is a known representation of $P_n(x)$.
\end{remark}

\section{Conclusion}

In this work, we have presented a unified treatment of a weighted Catalan convolution using hypergeometric representations and probabilistic interpretations. The hypergeometric closed form provides an explicit and tractable formula for the parametric sum $S_n(a)$, while the probabilistic viewpoint interprets it as a weighted convolution of return probabilities of independent random walks. The Narayana refinement further elucidates the combinatorial structure by accounting for peak distributions in Dyck paths. The integral representation suggests a connection with orthogonal polynomials and spectral measures, which may provide an alternative route to further generalizations. Collectively, these perspectives reveal the underlying holonomic structure and suggest that similar techniques can be applied to other convolutions of combinatorial sequences. Our methods demonstrate the interplay between combinatorial identities, hypergeometric analysis, and probabilistic reasoning, offering a versatile framework for future investigations.

\vspace{5mm}

\noindent \textbf{Acknowledgements}

\vspace{5mm}

I would like to thank Johan M. Ashfaque for pointing out an error in Theorem 2.1 in the previous version of the manuscript.

\end{document}